\documentclass[11pt]{amsart}
\usepackage{amsmath,amssymb,amscd,amsthm,amsxtra}
\usepackage{latexsym}

\usepackage[dvips]{graphicx}        

\newtheorem{theorem}{Theorem}[section]
\newtheorem{lemma}{Lemma}[section]
\newtheorem{proposition}{Proposition}[section]
\newtheorem{remark}{Remark}[section]

\newtheorem{definition}{Definition}[section]
\newtheorem{corollary}{Corollary}[section]

\newcommand{\e}{\varepsilon}

\newcommand{\wh}{\widehat}

\newcommand{\Id}{{\bf 1}}

\newcommand{\bn}{{\bf n}}
\newcommand{\bm}{{\bf m}}
\newcommand{\bl}{{\bf l}}
\newcommand{\bx}{{\bf x}}
\newcommand{\bk}{{\bf k}}

\begin{document}
\title{Discrete Fourier Restriction associated with Schr\"odinger equations}
\author{Yi Hu}
\address{
Yi Hu\\
Department of Mathematics\\
University of Illinois at Urbana-Champaign\\
Urbana, IL, 61801, USA}
\email{yihu1@illinois.edu}

\author{Xiaochun Li}

\address{
Xiaochun Li\\
Department of Mathematics\\
University of Illinois at Urbana-Champaign\\
Urbana, IL, 61801, USA}

\email{xcli@math.uiuc.edu}

\thanks{ This work was partially supported by an NSF grant DMS-0801154}

\begin{abstract} In this paper, we present a different proof on the discrete Fourier restriction. 
The proof recovers Bourgain's level set result on Strichartz estimates associated with 
Schr\"odinger equations on torus. Some sharp estimates on $L^{\frac{2(d+2)}{d}}$ norm of
certain exponential sums in higher dimensional cases are established. As an application, 
we show that some discrete multilinear maximal functions are bounded on $L^2(\mathbb Z)$.  
\end{abstract} 

\maketitle

\section{Introduction}
\setcounter{equation}0

In this paper, we consider discrete Fourier restriction problems associated with Schr\"odinger 
equations. More precisely, for any given $N\in\mathbb N$, let $S_{d, N}$ stand for the set
$$
 \left\{(n_1,\cdots, n_d)\in \mathbb{Z}^d: |n_j|\leq N,\ 1\leq j\leq d \right\}\,.
$$
For $p>1$, let $A_{p, d, N}$ represent the best constant satisfying 
\begin{equation}\label{1}
\sum_{\mathbf{n}\in S_{d,N}}\left|\widehat{f}(\mathbf{n}, |\mathbf{n}|^2)\right|^2\leq A_{p,d,N}\|f\|_{p'}^2\,,
\end{equation}
where $\mathbf{n}=(n_1, \cdots, n_d)\in S_{d, N}$, $ |\mathbf{n}|=\sqrt{n_1^2+ \cdots +n_d^2}$, $f$ is any $L^{p'}$-function on $\mathbb T^{d+1}$, $\widehat f$ stands for Fourier transform of periodic function $f$ on $\mathbb T^{d+1}$, and $p'= p/(p-1)$. \\

A harmonic analysis method was introduced by Bourgain \cite{B1} to obtain 
\begin{equation}\label{B-est}
A_{p,d,N}\leq C N^{d-\frac{2(d+2)}{p}+\e } \,{\rm for}\,\,\, p> \frac{2(d+4)}{d}\,.
\end{equation}
It was conjectured by Bourgain in \cite{B1} that 
\begin{equation}
A_{p,d,N}\leq
\begin{cases}
C_pN^{d-\frac{2(d+2)}{p}+\e}\quad &\text{for}\ \  p\geq \frac{2(d+2)}{d}\\
C_p                            &\text{for}\ \  2\leq p<\frac{2(d+2)}{d}
\end{cases}.
\end{equation}
The understanding of this conjecture is still incomplete. For instance, the desired upper bounds
for $A_{5, 1, N}$, $A_{3, 2, N}$ or $A_{\frac{2(d+2)}{d}, d, N}$ for $d\geq 3$ are not yet obtained. 
The most crucial estimate established by Bougain in \cite{B1} is certain (sharp) level set estimate.
In this paper we provide a different proof of the level set estimate.  \\

These problems arise from the study of periodic nonlinear Schr\"odinger equations:
\begin{equation}\label{Sch}
\begin{cases}
\Delta_\mathbf{x}u+i\partial_tu+u|u|^{p-2}=0\\
u(\mathbf{x},0)=u_0(\mathbf{x})
\end{cases}.\end{equation}
Here $\mathbf{x}=(x_1, \cdots, x_d)\in\mathbb{T}^d$, and $u(\mathbf{x},t)$ is a function of $d+1$ variables which is periodic in space.  The corresponding Strichartz estimate is the inequality seeking for the best constant $K_{p, d, N}$ 
satisfying 
\begin{equation}\label{Sest}
\left\|\sum_{\mathbf{n}\in S_{d,N}}a_\mathbf{n}e^{2\pi i(\mathbf{n}\cdot\mathbf{x}+|\mathbf{n}|^2t)}\right\|_{L^p(\mathbb{T}^{d+1})}\leq K_{p,d,N}\left(\sum_\mathbf{n}|a_\mathbf{n}|^2\right)^{1/2},\end{equation}
where $\{a_\mathbf{n}\}$ is a sequence of complex numbers.   
The restriction estimate (\ref{1}) is essentially the Strichartz estimate because  
\begin{equation}\label{3}
 K_{p, d, N}\sim \sqrt{A_{p, d, N}}\, 
\end{equation}
follows easily by duality. 
\\

The Duhamel's principle allows us to represent the differential equation as an integral equation
$$u(\mathbf{x},t)=e^{it\Delta}u_0(\mathbf{x})+i\int_{0}^te^{i(t-\tau)\Delta}\left(|u(\mathbf{x},\tau)|^{p-2}u(\mathbf{x},\tau)\right)d\tau\,.$$
Applying Picard's iteration and the Strichartz estimate (\ref{Sest}), Bourgain in \cite{B1} obtained local (global) well-posedness of the Schr\"odinger equations (\ref{Sch}). Hence, the discrete restriction problems are crucial to study the dispersive equations on torus. Moreover, they are closely related to Vinogradov mean value conjecture on exponential sums, which is very interesting and important in additive number theory.  \\

Let us introduce Vinogradov's mean value in order to see more clearly the connection between additive number theory and
discrete Fourier restriction.
For any given  polynomial $P(x, \alpha_1, \cdots, \alpha_d)=\sum_{j=1}^k \alpha_j x^j$ for $\alpha_1, \cdots, \alpha_k\in\mathbb T$, the mean value $J_{k}(N, b)$ is defined by
$$
J_{k}(N,b)=\int_{\mathbb T^k}\left|\sum_{n=1}^Ne^{2\pi i P(n, \alpha_1, \cdots, \alpha_k)} \right|^{2b} d\alpha_1\cdots d\alpha_k\,.
$$
The Vinogradov mean value conjecture asked the following question.
For positive integers $k$ and $b$, is it true that 
\begin{equation}\label{JK}
J_k(N, b)\leq C_{k,b, \varepsilon}(N^{b+\varepsilon} + N^{2b-\frac{k(k+1)}{2}+\varepsilon})\,?
\end{equation} 
Vinogradov invented a method (now called Vinogradov method) to establish some partial results on the 
mean value conjecture, and then utilize these partial results for exponential sums to gain new pointwise estimates,
which can not be done via Weyl's classical squaring method. One of main points in Vinogradov's method is that
pointwise estimates of the exponential sums follow from the suitable upper bound of the mean value.  Despite
many brilliant mathematicians devoted considerable time and energy to this conjecture, only $k=2$ case is completely
settled, and the conjecture is also answered affirmatively for cubic polynomials provided $b >8$ due to Hua's work.\\

In terms of the language of discrete restriction, Vinogradov's mean value conjecture can be rephrased as 
a statement asking whether the following inequality is true:
\begin{equation}\label{restrictionK}
\sum_{n=1}^N \left| \wh f(n, \cdots, n^k) \right|^2\leq C N^{1-\frac{k(k+1)}{p}+\e}  \|f\|^2_{p'}
\end{equation}
for $p\geq k(k+1)$. Of course, (\ref{restrictionK}) is apparently harder. In fact, (\ref{restrictionK}) implies 
the conjecture. But the conjecture only yields some partial results for (\ref{restrictionK}). It will 
be very interesting if the equivalence of (\ref{JK}) and (\ref{restrictionK}) could be established. \\

Despite the  overwhelming difficulty of (\ref{restrictionK}), we pose a relatively simple question here.
Let $k\geq 3$ be a positive integer. Suppose $p\geq 2(k+1)$. Is it true that 
\begin{equation}\label{res-d3}
  \sum_{n=1}^N\left |\wh f(n, n^k)\right|^2\leq  CN^{1-\frac{2(k+1)}{p} +\e } \|f\|_{p'}^2\,? 
\end{equation}
This question is essentially about the Strichartz estimates associated with higher order dispersive equations. Bourgain's proof on (\ref{B-est}) is based on three ingredients: Weyl's sum estimates, Hardy-Littlewood circle method, and Tomas-Stein's restriction theorem. It is difficult to employ Bourgain's method for (\ref{res-d3}). Hence 
we are forced to seek a method, which can be adjusted to handle the higher order polynomials like
$ax+bx^k$.  This is our main motivation. In this paper, we present a different proof of (\ref{B-est}).
This paper is our first paper on the discrete restriction. In the subsequent papers, we will 
modify this method to obtain an affirmative answer to (\ref{res-d3}) for $p$ large enough and then provide 
applications on the corresponding nonlinear dispersive equations.\\

Our first theorem is about weighted 
restriction estimates, which deal with the large $p$ cases of (\ref{1}). Moreover, there is 
no $\e$ required in the upper bound that we obtain.  

\begin{theorem}\label{thm1}
For any $\sigma >0$, any $d\in\mathbb N$, and any $p > \frac{4(d+2)}{d}$, there exists a constant $C$ independent of $N$ such that 
\begin{equation}\label{weight}
\sum_{\bn\in \mathbb Z^d } e^{-\frac{\sigma |\bn|^2}{N^2}} \left| \widehat {f}(\bn\,,  |\bn|^2)\right|^2 \leq C N^{d-\frac{2(d+2)}{p}}\|f\|_{p'}^2\,,
\end{equation}
for all $f\in L^{p'}(\mathbb T^{d+1})$. 
\end{theorem}

Theorem \ref{thm1} yields (\ref{B-est}) for large $p$ immediately. The proof of Theorem \ref{thm1} presented in
Section \ref{largeP} is very straightforward. The tool we use is Hardy-Littlewood circle method. The decay factor $e^{-\sigma|\bn|^2/N^2}$ makes it possible to calculate $L^p$ norm of the kernel restricted to major arcs or minor arcs. \\

For small $p$ cases, we need a new level set estimate, which implies Bourgain's level set estimate (see Corollary
\ref{cor1}). Its proof relies on a decomposition of the kernel,
which is a sum of a $L^\infty$ function and a function with bounded Fourier transform (see Proposition 
 \ref{Prop1}).

\begin{theorem}\label{thm2}
Suppose that 
$F $ is a periodic function on $\mathbb T^{d+1}$ given by
\begin{equation}
 F(\bx, t) = \sum_{\bn \in S_{d, N}} a_{\bn} e^{2\pi i \bn \cdot \bx} e^{2\pi i |\bn|^2 t}\,,
\end{equation}
where $\{a_\bn\}$ is a sequence with $\sum_\bn |a_\bn|^2 =1$ and $(\bx, t)\in \mathbb T^d\times \mathbb T  $.  For any $\lambda >0$, let
$$ E_\lambda =  \left\{ (\bx, t)\in \mathbb T^{d+1}: |F(\bx, t)|>\lambda\right\}\,. $$
Then for any positive number $Q$ satisfying $ Q\geq N$, 
\begin{equation}\label{levelsetest}
 \lambda^2 \left| E_\lambda\right|^2 \leq C_1Q^{d/2}\left| E_\lambda\right|^2 + \frac{C_2N^\e}{Q}\left| E_\lambda\right|\,
\end{equation}
holds for all $\lambda$.  Here $C_1$ and $C_2$ are constants independent of $N$ and $Q$.  

\end{theorem}

Applying Theorem \ref{thm2}, we can easily obtain the following corollaries, which were proved by Bourgain in \cite{B1} in a different way. The details will appear in Section \ref{level}.

\begin{corollary}\label{cor1}
If $\lambda \geq C N^{d/4}$ for some suitably large constant $C$, 
then the level set defined in Theorem \ref{thm2} satisfies 
$$
 |E_\lambda|\leq C_1 N^\varepsilon \lambda^{-\frac{2(d+2)}{d}} \,. 
$$
\end{corollary}

\begin{corollary}\label{cor2}
\begin{equation}\label{estKpdN2}
K_{p,d, N}\leq C_\varepsilon N^{\frac{d}{2}-\frac{d+2}{p}+\varepsilon}\,\,\,{\rm if} \,\,\, p > \frac{2(d+4)}{d}
\end{equation}
\end{corollary}

\begin{remark}
Corollary \ref{cor2} clealy yields (\ref{B-est}) because $K_{p,d,N}\sim \sqrt{A_{p,d,N}} $.
Moreover, the tiny positive number $\varepsilon$ in (\ref{estKpdN2}) can be removed. Clearly from Theorem \ref{thm1}, we see immediately that the 
$\varepsilon$ is superfluous for larger $p$. For $\frac{2(d+4)}{d}< p \leq \frac{4(d+2)}{d} $, Bourgain 
in \cite{B1} succeeded in removing the $\varepsilon$ via a delicate interpolation argument.  
At the moment we were writing this paper,  a new paper \cite{B4} posed by Bougain shows that the lower bound of $p$ 
can be improved to be $\frac{2(d+3)}{d}$ by a multi-linear restriction theory. 
\end{remark}

Moreover, Theorem \ref{thm2} implies the following recurrence relation  on $K_{p,d,N}$ in the sense of inequality.

\begin{corollary}\label{cor3}
For $p>2$, we have
\begin{equation}\label{re}
 K_{p, d, N}^p \leq C N^{d}K^{p-2}_{p-2, d, N} + CN^{\frac{dp}{2}-d-2+\e}\,. 
\end{equation}
Here $C$ is independent of $N$.
\end{corollary}

These three corollaries will be proved in Section \ref{level}. 
Carrying on the idea used in the proof of Theorem \ref{thm2}, we can get the following theorem.

\begin{theorem}\label{thm3}
Let $N_1, \cdots, N_d\in\mathbb N$ and $S_{N_1,\cdots, N_d}$ be defined by 
\begin{equation}\label{defofS}
S_{N_1, \cdots, N_d}(\bx, t) = \sum_{\bn\in S(N_1, \cdots, N_d)} e^{2\pi i \bn \cdot \bx} e^{2\pi i |\bn|^2 t} \,.
\end{equation}
where $S(N_1, \cdots, N_d)$ is given by
\begin{equation}
S(N_1, \cdots, N_d)=\{\bn=(n_1, \cdots, n_d)\in \mathbb Z^d: |n_j|\leq N_j\,\,\, {\rm for}\,\,\, {\rm all}\,\, j\in\{1, \cdots, d\}  \}\,.
\end{equation}
For any $\varepsilon>0$, there exists a constant $C$ independent of $N$ such that
\begin{equation}\label{estS}
\left\|S_{N_1, \cdots, N_d}\right\|_{\frac{2(d+2)}{d}}\leq 
C \left( N_1\cdots N_d \right)^{\frac{d}{2(d+2)}} \max\{N_1,\cdots, N_d\}^{\frac{d}{d+2} +\e}\,. 
\end{equation}
\end{theorem}

Observe that if $N_1=\cdots=N_d=N$, (\ref{estS}) implies that
\begin{equation}
 \left\| \sum_{\bn\in S_{d, N}}e^{2\pi i \bn \cdot \bx} e^{2\pi i |\bn|^2 t} \right\|_{2(d+2)/d}
\leq N^{\frac{d}{2}+\e}\,,
\end{equation}
that is, 
\begin{equation}
\left\| \sum_{\bn\in S_{d, N}}a_{\bn} e^{2\pi i \bn \cdot \bx} e^{2\pi i |\bn|^2 t} \right\|_{2(d+2)/d}
\leq N^\e \left(\sum_\bn |a_\bn|^2 \right)^{1/2}
\end{equation}
provided  $a_\bn=1 $ for all $\bn$.  If the conditions $a_\bn=1$ for all $\bn$ could be 
removed, then the Bourgain conjecture would be solved for all $p$'s not less than the critical index
$2(d+2)/d$.

Theorem \ref{thm3} has a direct application to some multi-linear maximal functions, related to maximal ergodic theorem, for instance, to pointwise convergence of  
the non-conventional bi-linear average  
\begin{equation*}
N ^{-1} \sum _{n=1} ^{N} f_1 (T ^{n}) f_2 (T ^{n ^2 })\,,
\end{equation*}
where $T$ is a measure preserving transformation on a probability space $(X, \mathcal A, \mu)$.
This application will appear in Section {\ref{appl1}}.\\

\section{Large $p$ Cases}\label{largeP}

In this section we provide a proof of Theorem \ref{thm1}. All we need to employ is the Hardy-Littlewood circle method.
Observe that for large $p$, $A_{p, d, N}\leq C N^{d-\frac{2(d+2)}{p}}$ follows immediately by noticing 
$$
\sum_{\mathbf{n}\in S_{d,N}}\left|\widehat{f}(\mathbf{n}, |\mathbf{n}|^2)\right|^2\leq e^{\sigma d}\sum_{\mathbf{n}\in S_{d,N}}e^{-\frac{\sigma |\mathbf{n}|^2}{N^2}}\left|\widehat{f}(\mathbf{n}, |\mathbf{n}|^2)\right|^2\leq e^{\sigma d}\sum_{\mathbf{n}\in\mathbb{Z}^d}e^{-\frac{\sigma |\mathbf{n}|^2}{N^2}}\left|\widehat{f}(\mathbf{n}, |\mathbf{n}|^2)\right|^2.
$$  
Thus Theorem \ref{thm1} yields the desired upper bounds of $A_{p,d,N}$ for large $p$ cases.
Here the decay factor $e^{-\frac{\sigma |\mathbf{n}|^2}{N^2}}$ will make our calculation much easier. 
The key idea is to decompose the circle into arcs (called major arcs and minor arcs) and then estimate $L^p$ norm 
of the corresponding kernel over each arcs.\\

First we present some technical lemmas. In order to introduce the major arcs, we should state Dirichlet principle. 

\begin{lemma} {\rm({\bf Dirichlet Principle})}
  For any given $N\in \mathbb N$ and any $t\in (0,1]$, there exist $a, q\in\mathbb N$, $1\leq q\leq N$, $1\leq a\leq q$, $(a, q)=1$, such that $\left|t-\frac{a}{q}\right|\leq\frac{1}{Nq}.$
\end{lemma}

This principle can be proved by utilizing the pigeonhole principle or by the Farey dissection of order $N$.
For any integer $q$, define $\mathcal P_q$ by 
$$
 \mathcal P_q = \{ a\in \mathbb Z:  1\leq a\leq q, (a, q)=1\}\,,
$$
 and for any $a\in \mathcal P_q $, set the interval $J_{a/q}$ by 
$J_{a/q} = (\frac{a}{q}- \frac{1}{Nq}, \frac{a}{q} + \frac{1}{Nq})$.  If $q< N/10$, the interval $J_{a/q}$ 
is called a major arc, otherwise, a minor arc.  Clearly we can partition $(0,1]$ into a union of major arcs and
minor arcs, that is,
$$
 (0,1]=\bigcup_{1\leq q\leq N, a\in \mathcal P_q} J_{a/q} = \mathcal M_1 \cup \mathcal M_2\,.
$$ 
Here $\mathcal M_1$ is the collection of all major arcs and $\mathcal M_2$ is the union of all minor arcs.

\begin{lemma}\label{lem2}
Let ${\bf 1}_A$ denote the indicator function of a measurable set $A$. Then 
\begin{equation}
\left\|\sum_{J\in \mathcal M_1} {\bf 1}_{J} \right\|_\infty +  \left\|\sum_{J\in \mathcal M_2} {\bf 1}_{J} \right\|_\infty  \leq 100 \,.
\end{equation}
\end{lemma}

\begin{proof}
It is easy to see that all major arcs are disjoint. Thus it suffices to prove that 
$$\left\|\sum_{J\in \mathcal M_2} {\bf 1}_{J} \right\|_\infty \leq 80\,.
$$
In fact, for any given minor arc $J_{a_0/q_0}$, let $\mathcal Q$ denote the collection of 
all rational numbers $a/q$'s such that each $J_{a/q}$ is a minor  arc and 
there is a common point of $J_{a_0/q_0}$ and all $J_{a/q}$'s. We should prove that the cardinality of 
$\mathcal Q$ is less than $40$.  Notice that for any $a/q\in\mathcal Q$, 
$$
 \left| \frac{a_0}{q_0} -\frac{a}{q} \right| < \frac{1}{Nq_0}+\frac{1}{Nq}\,.
$$
This implies that $|a_0 q- aq_0|< 2$. Since $a_0q-aq_0\in\mathbb Z$, we conclude that
either $a_0 q- aq_0=-1$ or $a_0 q-a q_0=1$ if $a/q\neq a_0/q_0$. Hence if $a/q\neq a_0/q_0$, $a/q\in \mathcal Q$
must satisfy the diophantine equation $a_0 x -q_0 y=-1$ or $ a_0 x -q_0 y=1$ with $|x|\leq N$. 
The general solution of the diophantine 
equation is $ x= x_0+ q_0 k$ and $y=y_0+ a_0 k$ for all $k\in\mathbb Z$ and any given particular solution $(x_0, y_0)$.
Then $|kq_0|\leq 2N$. By $q_0\geq N/10$, we have $|k|\leq 20$. Thus the number of solutions of either diophantine equation 
is no more than $40$. This completes the proof. 

\end{proof}

\begin{remark}
Lemma \ref{lem2} is about the finite overlapping property of minor arcs. The reason why we use this lemma is that
we try to only calculate $L^p$ norm of the kernel restricted to each arc. Of course, this is not necessarily needed.
 An alternative way, which is very classic, is to obtain $L^\infty$
norm for the kernel restricted to the union of minor arcs, and then to find $L^p$ norm of the kernel on each major arc.    
\end{remark}

Let $K_\sigma$ be a kernel defined by
\begin{equation}\label{defKsi}
K_{\sigma}(\bx, t) = \sum_{\bn\in\mathbb Z^d} e^{-\frac{\sigma|\bn|^2}{N^2}} e^{2\pi i |\bn|^2 t} e^{2\pi i \bn \cdot \bx }\,. 
\end{equation}

We set $K_{a/q}$ to be
\begin{equation}
K_{a/q}(\bx, t) =  K_\sigma(\bx, t) {\bf 1}_{J_{a/q}}(t)\,.
\end{equation}
The following lemma gives an upper bound for $L^p$ norm of $K_{a/q}$. 

\begin{lemma}\label{lem3}
For any integer $1\leq q\leq N$, any integer $a\in\mathcal P_q$ and any $p>\frac{2(d+1)}{d} $,
\begin{equation}
\left\|K_{a/q}\right\|_{p}\leq \frac{C N^{d-\frac{d+2}{p}}}{q^{\frac{d}{2}- \frac{d}{p}}}\,.
\end{equation}
\end{lemma}

\begin{proof}

For any given $t\in J_{a/q}$, let $\beta = t-\frac{a}{q}$ and write $\bn= \bk q+\bl $. Here 
$\bl\in \mathbb Z^d_q =\{(l_1,\cdots, l_d): l_j\in \mathbb Z_q\}$. Then we have
$$
K_\sigma(\mathbf{x},t) =\sum_{\mathbf{k}\in\mathbb{Z}^d}\sum_{\mathbf{l}\in\mathbb{Z}_q^d}
e^{-\frac{\sigma |\mathbf{k}q+\mathbf{l}|^2}{N^2}}e^{2\pi i(\mathbf{k}q+\mathbf{l})\cdot\mathbf{x}}e^{2\pi i|\mathbf{k}q+\mathbf{l}|^2(\frac{a}{q}+\beta)}\,.
$$
Interchanging the sums, we represent the kernel as 
$$
K_\sigma(\bx, t)=\sum_{\mathbf{l}\in\mathbb{Z}_q^d}e^{2\pi i|\mathbf{l}|^2\frac{a}{q}}\sum_{\mathbf{k}\in\mathbb{Z}^d}e^{-|\mathbf{k}q+\mathbf{l}|^2(\frac{\sigma}{N^2} -2\pi i\beta)}e^{2\pi i(\mathbf{k}q+\mathbf{l})\cdot\mathbf{x}}.
$$
Applying Poisson summation formula to the inner sum, 
we have 
$$
\sum_{\mathbf{k}\in\mathbb{Z}^d}e^{-|\mathbf{k}q+\mathbf{l}|^2(\frac{\sigma}{N^2} -2\pi i\beta)}e^{2\pi i(\mathbf{k}q+\mathbf{l})\cdot\mathbf{x}}= 
\sum_{\mathbf{k}\in\mathbb{Z}^d}\left(\frac{\sqrt{\pi}}{q\sqrt{\frac{\sigma}{N^2}-2\pi i\beta}}\right)^de^{2\pi i\frac{\mathbf{l}\cdot\mathbf{k}}{q}}e^{-\frac{\pi^2|\mathbf{x}-\frac{\mathbf{k}}{q}|^2}{\frac{\sigma}{N^2}-2\pi i\beta}}
$$
Henceforth, the kernel can be written as 
\begin{equation}
K_\sigma(\mathbf{x},t)
=\left(\frac{\sqrt{\pi}}{q\sqrt{\frac{\sigma}{N^2}- 2\pi
i\beta}}\right)^d\sum_{\mathbf{k}\in\mathbb{Z}^d}e^{-\frac{\pi^2|\mathbf{x}-\frac{\mathbf{k}}{q}|^2}{\frac{\sigma}{N^2}-2\pi i\beta}}\sum_{\mathbf{l}\in\mathbb{Z}_q^d}e^{2\pi i|\mathbf{l}|^2\frac{a}{q}}e^{2\pi i\mathbf{l}\cdot\frac{\mathbf{k}}{q}}.
\end{equation}
From the well-known result on the upper bound of the Gauss sum, it follows that
$$
\left|\sum_{\mathbf{l}\in\mathbb{Z}_q^d}e^{2\pi i|\mathbf{l}|^2\frac{a}{q}}e^{2\pi i\mathbf{l}\cdot\frac{\mathbf{k}}{q}}\right|
\leq (2q)^{d/2}\,.
$$
Thus by inserting the absolute value, the kernel can be majorized by 
$$
|K_\sigma(\mathbf{x},t)| \leq
\frac{(2\pi)^{d/2}}{q^\frac{d}{2}\left(\frac{\sigma^2}{N^4}+4\pi^2\beta^2\right)^\frac{d}{4}}
\sum_{\mathbf{k}\in\mathbb{Z}^d}e^{-\frac{\pi^2|\mathbf{x}-\frac{\mathbf{k}}{q}|^2\frac{\sigma}{N^2}}{\frac{\sigma^2}{N^4}+4\pi^2\beta^2}}  \,.
$$
Integrating $|K_\sigma|^p$ on each arc $J_{a/q}$, we obtain that
\begin{eqnarray*}
\left\|K_{a/q} \right\|_{p}^p & \leq  & \int_{|\beta|\leq \frac{1}{Nq}} \int_{\mathbb T^d} \frac{(2\pi)^{dp/2}}{q^\frac{dp}{2}\left(\frac{\sigma^2}{N^4}+4\pi^2\beta^2\right)^\frac{dp}{4}}
\left|\sum_{\mathbf{k}\in\mathbb{Z}^d}e^{-\frac{\pi^2|\mathbf{x}-\frac{\mathbf{k}}{q}|^2\frac{\sigma}{N^2}}{\frac{\sigma^2}{N^4}+4\pi^2\beta^2}}\right|^p d\bx d\beta \\
&= &     \int_{|\beta|\leq \frac{1}{Nq}}  \frac{(2\pi)^{dp/2}}{q^\frac{dp}{2}\left(\frac{\sigma^2}{N^4}+4\pi^2\beta^2\right)^\frac{dp}{4}}  \left( \int_{0}^1
\left|\sum_{k\in\mathbb{Z}}e^{-\frac{\pi^2|x-\frac{k}{q}|^2\frac{\sigma}{N^2}}{\frac{\sigma^2}{N^4}+4\pi^2\beta^2}}\right|^p dx \right)^d  d\beta        \,.
\end{eqnarray*}
Notice that for $|\beta|\leq \frac{1}{Nq}$ and $q\leq N $,
$$
 \frac { \frac{\sigma}{q^2 N^2} }{ \frac{\sigma^2}{N^4}+4\pi^2\beta^2} \geq  C_\sigma\,.
$$
This yields that
$$
\sum_{k\in\mathbb{Z}}e^{-\frac{\pi^2|x-\frac{k}{q}|^2\frac{\sigma}{N^2}}{\frac{\sigma^2}{N^4}+4\pi^2\beta^2}}\leq C_\sigma\,.
$$
For $p > \frac{2(d+1)}{d}$, we estimate $L^p$ norm of $K_{a/q}$ by
$$
\left\|K_{a/q} \right\|_{p}^p \leq  \int_{|\beta|\leq \frac{1}{Nq}}  \frac{(2\pi)^{dp/2}}{q^\frac{dp}{2}\left(\frac{\sigma^2}{N^4}+4\pi^2\beta^2\right)^\frac{dp}{4}}  \left( \int_{0}^1
\sum_{k\in\mathbb{Z}}e^{-\frac{\pi^2|x-\frac{k}{q}|^2\frac{\sigma}{N^2}}{\frac{\sigma^2}{N^4}+4\pi^2\beta^2}} dx \right)^d  d\beta        \,,
$$
which can be bounded by
$$
\int_{|\beta|\leq \frac{1}{Nq}}  \frac{C(2\pi)^{\frac{dp}{2}} N^d}{q^{\frac{dp}{2}-d}\left(\frac{\sigma^2}{N^4}+4\pi^2\beta^2\right)^{\frac{dp}{4}-\frac{d}{2} }   }  
d\beta  \leq  \frac{C N^{dp-d-2}}{q^{\frac{dp}{2}-d}}  \,.    
$$
Therefore, we finish our proof.
\end{proof}

\begin{lemma}\label{lem4}
For $p>\frac{2(d+2)}{d}$, 
\begin{equation}
 \left\|K_\sigma\right\|_p \leq C_{p, \sigma} N^{d-\frac{d+2}{p}}\,.  
\end{equation}
\end{lemma}

\begin{proof}
By Lemma \ref{lem2} and Lemma \ref{lem3}, we have that 
$$
\left\|K_\sigma\right\|^p_p\leq C\sum_{q=1}^N\sum_{a\in\mathcal P_q} \left\| K_{a/q}\right\|_p^p
 \leq  C\sum_{q=1}^N\sum_{a\in\mathcal P_q} \frac{N^{dp-d-2}}{q^{\frac{dp}{2}-d}}  \leq CN^{dp-d-2}\,,
$$
which yields Lemma \ref{lem4}.
\end{proof}

We now return to the proof of Theorem \ref{thm1}.  Indeed, observe that
$$
 \sum_{\bn\in \mathbb Z^d } e^{-\frac{\sigma |\bn|^2}{N^2}} \left| \widehat {f}(\bn\,,  |\bn|^2)\right|^2 
= \left\langle  K_\sigma *f, f \right\rangle\,. 
$$
Applying H\"older's inequality and then Hausdorff-Young's inequality  on convolution, we get
$$
\left\langle  K_\sigma *f, f \right\rangle \leq \|K_\sigma\|_{p/2}\|f\|_{p'}^2\,. 
$$
Since $p>\frac{4(d+2)}{d}$, we employ Lemma \ref{lem4} to conclude Theorem \ref{thm1}. \\

\section{Level Set Estimates}\label{level}

In this section, we provide a proof of Theorem \ref{thm2}. Theorem \ref{thm2} can be utilized for handling 
small $p$ cases.

First, we state an arithmetic result. 

\begin{lemma}\label{lem6}
For any integer $Q\geq 1$ and any integer $n\neq 0$, and any $\varepsilon>0$, 
$$\sum_{ Q\leq q < 2Q}\left|\sum_{a\in\mathcal{P}_q}e^{2\pi i\frac{a}{q}n}\right|\leq C_\varepsilon
 d(n, Q) Q^{1+\varepsilon}\,.$$
Here $d(n, Q)$ denotes the number of divisors of $n$ less than $Q$ and $C_\varepsilon$ is a constant independent of
$Q, n$. 
\end{lemma}

Lemma \ref{lem6} can be proved by observing that the arithmetic function defined by 
$ f(q)=\sum_{a\in\mathcal{P}_q}e^{2\pi i\frac{a}{q}n}$ is multiplicative, and then utilize the prime 
factorization for $q$ to conclude the lemma. The details can be found in \cite{B1}. \\

We now state a proposition crucial to our proof. 

\begin{proposition}\label{Prop1}
For any given positive number $Q$ with $N\leq Q\leq N^2$,
the kernel $K_\sigma$ given by (\ref{defKsi}) can be decomposed into $K_{1, Q} + K_{2, Q}$ such that 
\begin{equation}\label{K1}
\|K_{1, Q}\|_\infty \leq C_1 Q^{\frac{d}{2}}\,.
\end{equation}
and
\begin{equation}\label{K2}
\|\widehat{K_{2, Q}}\|_{\infty} \leq \frac{C_2 N^\varepsilon }{Q}\,.
\end{equation}
Here the constants $C_1, C_2$ are independent of $Q$ and $N$. 
\end{proposition}

\begin{proof}
We can assume that $Q$ is an integer, since  otherwise we can take the integer part of $Q$. 
For a standard bump function $\varphi$ supported on $[1/200, 1/100]$, we set 
\begin{equation}\label{defofPhi}
\Phi(t) = \sum_{ Q\leq q < 2Q}\sum_{a\in\mathcal P_q}\varphi\left(\frac{t-a/q}{1/q^2}\right)\,.
\end{equation}
Clearly $\Phi$ is supported on $[0,1]$. We can extend $\Phi$ to other intervals periodically to obtain 
a periodic function on $\mathbb T$. For this periodic function generated by $\Phi$, we still use $\Phi$ to
denote it.  Then it is easy to see that
\begin{equation}
\widehat{\Phi}(0)=\sum_{q\sim Q}\sum_{a\in\mathcal{P}_q}\frac{\mathcal{F}_{\mathbb R}{\varphi}(0)}{q^2}=\sum_{q\sim Q}\frac{\phi(q)}{q^2}\mathcal{F}_{\mathbb R}{\varphi}(0) \,
\end{equation}
is a constant independent of $Q$. Here $\phi$ is Euler's totient function, and $\mathcal{F}_{\mathbb R}$ denotes Fourier transform of a function on $\mathbb R$. Also we have
\begin{equation}\label{Fest}
\widehat{\Phi}(k)= \sum_{q\sim Q}\sum_{a\in\mathcal P_q}\frac{1}{q^2} e^{-2\pi i \frac{a}{q}k} \mathcal F_{\mathbb R}
\varphi(k/q^2)\,. 
\end{equation}

We define that 
$$
K_{1, Q}(\bx, t)=\frac{1}{\widehat{\Phi}(0)}K_\sigma(\bx, t)\Phi(t), \,\,\,{\rm and}\,\,\, 
K_{2, Q} =  K_{\sigma} - K_{1, Q}\,. 
$$
We prove (\ref{K2}) first. In fact, write $\Phi$ as its Fourier series to get
$$
K_{2, Q}(\bx, t) = - \frac{1}{\widehat\Phi(0)} \sum_{k\neq 0}\widehat\Phi(k) e^{2\pi i k t} K_\sigma(\bx, t)\,.
$$
Thus its Fourier coefficient is
$$
\widehat{K_{2, Q}}(\bn, n_{d+1})= -\frac{e^{-\sigma|\bn|^2/N^2}}{\widehat\Phi(0)} \sum_{k\neq 0} 
  \widehat\Phi(k){\bf 1}_{\{ n_{d+1} = |\bn|^2 + k\} } (k) 
 \,.
$$
Here $\bn\in\mathbb Z^d$ and $n_{d+1}\in\mathbb Z$. This implies that 
$\widehat{K_{2, Q}}(\bn, n_{d+1}) =0$ if $ n_{d+1} = |\bn|^2 $, and if $n_{d+1}\neq |\bn|^2$, 
$$  \widehat{K_{2, Q}}(\bn, n_{d+1}) =  -\frac{e^{-\sigma|\bn|^2/N^2}}{\widehat\Phi(0)}
\widehat\Phi(n_{d+1}-|\bn|^2) \,.$$
Applying (\ref{Fest}) and Lemma \ref{lem6}, we estimate $\widehat{K_{2, Q}}(\bn, n_{d+1})$ by
$$
 \left|\widehat{K_{2, Q}}(\bn, n_{d+1})\right| \leq \frac{CN^\varepsilon}{Q}\,,
$$
since $N\leq Q\leq N^2$. Henceforth we obtain (\ref{K2}).\\

We now prove (\ref{K1}). Observe that $[\frac{a}{q}+\frac{1}{200q^2}, \frac{a}{q}+\frac{1}{100q^2}]$'s
are pairwise disjoint. Thus we can fix $q\sim Q$ and $a\in\mathcal P_q$ and try to obtain the upper bound 
of $K_{1, Q}$ restricted to $[\frac{a}{q}+\frac{1}{200q^2}, \frac{a}{q}+\frac{1}{100q^2}]$. Let $ \beta= t-\frac{a}{q}$. Hence we have $ |\beta|\sim 1/q^2 $ for $t\in [\frac{a}{q}+\frac{1}{200q^2}, \frac{a}{q}+\frac{1}{100q^2}] $.
As we did in the previous section, by Poisson summation formula, we have 
 $$
K_\sigma(\bx, t) = 
\left(\frac{\sqrt{\pi}}{q\sqrt{\frac{\sigma}{N^2}-2\pi
i\beta}}\right)^d\sum_{\mathbf{k}\in\mathbb{Z}^d}e^{-\frac{\pi^2|\mathbf{x}-\frac{\mathbf{k}}{q}|^2}{\frac{\sigma}{N^2}-2\pi i\beta}}\sum_{\mathbf{l}\in\mathbb{Z}_q^d}e^{2\pi i|\mathbf{l}|^2\frac{a}{q}}e^{2\pi i\mathbf{l}\cdot\frac{\mathbf{k}}{q}}\,.
$$
Hence for $|\beta|\sim 1/q^2$, we estimate 
$$
|K_\sigma(\bx, t)| \leq \frac{C}{q^{d/2}\left(\left(\frac{\sigma}{N^2}\right)^2+\beta^2\right)^\frac{d}{4}}
\sum_{\mathbf{k}\in\mathbb{Z}^d}
e^{-\pi^2\frac{\left|\frac{\mathbf{k}}{q}-\mathbf{x}\right|^2}{\left(\frac{\sigma}{N^2}\right)^2+\beta^2}\frac{\sigma}{N^2}}\,,
$$
which is bounded by
$$
 \frac{CN^{d}}{q^{d/2}}
\sum_{\mathbf{k}\in\mathbb{Z}^d}
e^{-\pi^2\frac{N^2}{\sigma}\left|\frac{\mathbf{k}}{q}-\mathbf{x}\right|^2} \leq  C_\sigma q^{d/2}\leq C_\sigma Q^{d/2}\,.
$$
This implies (\ref{K1}). Therefore we complete the proof. 

\end{proof}

We now start to prove Theorem \ref{thm2}. For the function $F$ and the level set $E_\lambda $ given in
Theorem \ref{thm2}, we define $f$ to be
$$
f(\bx, t) = \frac{{F(\bx, t)}}{ |F(\bx, t)|} {\bf 1}_{E_\lambda}(\bx, t)\,. 
$$
Clearly
$$
\lambda|E_\lambda|\leq \int_{\mathbb T^{d+1}} \overline{F(\bx, t)} f(\bx, t) d\bx dt\,.  
$$
By the definition of $F$, we get
$$
\lambda|E_\lambda|\leq \sum_{\bn\in S_{d, N}}\overline{a_\bn} \widehat{f}(\bn, |\bn|^2 )\,.
$$
Utilizing Cauchy-Schwarz's inequality, we have
$$
\lambda^2|E_\lambda|^2\leq \sum_{\bn\in S_{d, N}}\left| \widehat{f}(\bn, |\bn|^2)\right|^2 \,.
$$
The right hand side is bounded by
$$
 e^{\sigma d} \sum_{\bn}e^{-\frac{\sigma|\bn|^2}{N^2}} \left| \widehat{f}(\bn, |\bn|^2)\right|^2  = e^{\sigma d}\langle K_\sigma *f, f\rangle\,. 
$$
For any $Q$ with $N\leq Q\leq N^2$, we employ Proposition \ref{Prop1} to decompose the kernel $K_\sigma$. Then 
we have 
$$
\lambda^2|E_\lambda|^2\leq C_\sigma \left| \langle K_{1, Q} *f, f\rangle \right| 
   +  C_\sigma\left| \langle K_{2, Q} *f, f\rangle \right|  \,
$$
From (\ref{K1}) and (\ref{K2}), we then obtain 
$$
\lambda^2|E_\lambda|^2\leq C_1Q^{d/2}\|f\|_1^2 + \frac{C_2 N^\varepsilon}{Q}\|f\|_2^2
  \leq C_1Q^{d/2}|E_\lambda|^2 + \frac{C_2 N^\varepsilon}{Q}|E_\lambda|\,. 
$$
The case $Q\geq N^2$ is trivial since the level set $E_\lambda$ is empty if $\lambda > CN^{d/2}$.  
Therefore, we finish the proof of Theorem \ref{thm2}. \\

We now start to prove Corollary \ref{cor1} by using Theorem \ref{thm2}.  We should take $Q^{d/2}=
\frac{1}{2C_1}\lambda^2$, where $C_1$ is the constant stated in (\ref{levelsetest}). Since $Q\geq N$,
we need to restrict $\lambda > \sqrt{2C_1}N^{d/4}$.  Then $|E_\lambda|\leq CN^\varepsilon \lambda^{-2(d+2)/d}$
follows immediately from (\ref{levelsetest}). This completes the proof of Corollary \ref{cor1}.\\

To prove Corollary \ref{cor2}, write 
$$
\|F\|^p_p = C_p \int_0^\infty \lambda^{p-1}|E_\lambda| d\lambda \,,
$$
which equals to
$$
C_p \int_0^{CN^{d/4}} \lambda^{p-1}|E_\lambda| d\lambda +  C_p \int_{CN^{d/4}}^{\infty} \lambda^{p-1}|E_\lambda| d\lambda\,.
$$
Utilizing the trivial estimate $|E_\lambda|\leq C\lambda^{-2}$ for the first term and 
 employing Corollary \ref{cor1} for the second term, we then obtain, for $p>\frac{2(d+4)}{d}$, 
$$
\|F\|^p_p \leq C N^{\frac{dp}{2}-(d+2)+\varepsilon}\,
$$ 
as desired. Therefore the proof of Corollary \ref{cor2} is completed. \\

We now prove Corollary \ref{cor3}. Multiply (\ref{levelsetest}) by $\lambda^{p-3}$ to get, 
for $N\leq Q$, 
\begin{equation}\label{level3}
\lambda^{p-1}|E_\lambda| \leq C_1Q^{d/2}\lambda^{p-3}|E_\lambda| 
   + \frac{C_2N^\e}{Q}\lambda^{p-3}\,. 
\end{equation}
Integrating (\ref{level3}) in $\lambda$ from $0$ to $ CN^{d/2}$, we obtain that
\begin{equation}
\|F\|_p^p \leq C_1Q^{d/2}\|F\|_{p-2}^{p-2} + C_2\frac{N^{\frac{dp}{2}-d+\e}}{Q}\,.
\end{equation}
Taking $Q=N^2$, we then have
\begin{equation}
\|F\|_p^p \leq C_1N^d K_{p-2, d, N}^{p-2} + C_2N^{\frac{dp}{2}-d-2+\e}\,. 
\end{equation}
This finishes the proof of Corollary \ref{cor3}. \\

\section{Proof of Theorem \ref{thm3}}

In this section, we prove Theorem \ref{thm3} by carrying the similar idea shown in Section \ref{level}. 
We introduce a level set $G_\lambda$ for any $\lambda>0$ by setting,
\begin{equation}\label{defGLa} 
 G_\lambda = \left\{ (\bx, t)\in \mathbb T^d \times \mathbb T:  |S_{N_1,\cdots,N_d}(\bx, t)| > \lambda  \right\}\,.
\end{equation}
As we did in Section \ref{level}, let $f=\Id_{G_\lambda}S_{N_1,\cdots, N_d}/|S_{N_1, \cdots, N_d}|$ and we then have
\begin{equation}\label{estGLam}
\lambda|G_\lambda| \leq \sum_{\bn\in S(N_1, \cdots, N_d)} \widehat f(\bn, \bn^2) = \langle f_{N_1, \cdots, N_d}, S_{N_1, \cdots, N_d} \rangle\,,
\end{equation} 
where $ f_{N_1, \cdots, N_d}$ is a rectangular Fourier partial sum defined by 
\begin{equation}
f_{N_1, \cdots, N_d}(\bx, t) = \sum_{\substack{\bn\in S(N_1, \cdots, N_d)\\ |n_{d+1}|\leq d
 \max\{N_1, \cdots, N_d\}^2 } } \wh f(\bn, n_{d+1}) e^{2\pi \bn \cdot \bx } e^{2\pi i n_{d+1} t}\,.
\end{equation}
Here unlike what we did in Section \ref{level}, we do not use Cauchy-Schwarz inequality for the right hand side of (\ref{estGLam}). 
We actually need to get a decomposition of $S_{N_1,\cdots, N_d}$. Before we state this decomposition, we should include 
a famous result on Weyl's sum.

\begin{lemma}\label{Weyl}
Suppose $t$ is a real number satisfying 
$$\left|t-\frac{a}{q}\right|\leq \frac{1}{q^2}\,.$$
Here $a$ and $q$ are relatively prime integers. Then
\begin{equation}\label{Weylest}
\left| \sum_{n=1}^N e^{2\pi i (t n^2 + x n)}\right|\leq C{\rm max}\left\{ \frac{N}{\sqrt{q}},
  \sqrt{N \log q}, \sqrt{q\log q}  \right\}\,.
\end{equation}  
\end{lemma} 

The proof can be done by Weyl's squaring method. See \cite{Hua} or \cite{M} for details.

\begin{lemma}\label{lemThm3}
For any real number $Q$ with $\max\{N_1, \cdots, N_d\} \leq Q\leq \max\{N_1, \cdots, N_d\}^2$,  the function $S_{N_1, \cdots, N_d}$ defined in (\ref{defofS}) can be
written as a sum of $S_{1, Q}$ and $S_{2, Q}$, where $S_{1, Q}$ satisfies
\begin{equation}\label{estS1Q}
\|S_{1, Q}\|_\infty\leq  CQ^{d/2} (\log Q)^{d/2} 
\end{equation}
and $S_{2, Q}$ satisfies
\begin{equation}\label{estS2Q}
\|\widehat{S_{2, Q}}\|_{\infty}\leq \frac{C\max\{N_1, \cdots, N_d\}^\varepsilon}{Q}\,. 
\end{equation}
Here the constant $C$ is independent of $N_1, \cdots, N_d$ and $Q$. 
\end{lemma}   

\begin{proof}
Let $\Phi$ be the function defined in (\ref{defofPhi}). We then obtain 
\begin{equation}
S_{N_1, \cdots, N_d}= S_{1, Q} + S_{2, Q}\,,
\end{equation}
where $S_{1, Q}$ is given by
\begin{equation}
S_{1, Q}(\bx, t) = \frac{1}{\wh\Phi(0)}S_{N_1,\cdots, N_d}(\bx, t) \Phi(t)\,
\end{equation}
and $S_{2, Q}$ is 
\begin{equation}
S_{2, Q}= S_{N_1, \cdots, N_d}-S_{1, Q}\,. 
\end{equation}
(\ref{estS1Q}) follows immediately from (\ref{Weylest}). Notice that
$$
S_{2, Q}(\bx, t) = -\frac{1}{\wh\Phi(0)} \sum_{k\neq 0}\wh\Phi(k)e^{2\pi i kt}S_{N_1, \cdots, N_d}(\bx, t)\,. 
$$
(\ref{estS2Q}) follows by using Lemma \ref{lem6}, as we did in the proof of (\ref{K2}). 
Hence we finish the proof.
\end{proof}

We now return to the proof of Theorem \ref{thm3}. From (\ref{estGLam}) and Lemma \ref{lemThm3},
the level set $G_\lambda $ satisfies 
\begin{equation}
\lambda|G_\lambda|\leq |\langle f_{N_1,\cdots, N_d}, S_{1, Q}\rangle| + |\langle f_{N_1, \cdots, N_d}, S_{2, Q}\rangle |  \,, 
\end{equation}
which can be bounded by
\begin{equation}
C\left( Q^{d/2} (\log Q)^{d/2}\|f_{N_1, \cdots, N_d}\|_1 + \sum_{
\substack{\bn\in S(N_1, \cdots, N_d)\\ |n_{d+1}|\leq d
 \max\{N_1, \cdots, N_d\}^2 }   }
   \left| \wh{S_{2, Q}}(\bn, n_{d+1})\wh f(\bn, n_{d+1}) \right|  \right)\,.
\end{equation}
Thus from the fact that $L^1$ norm of Dirichlet kernel $D_N$ is comparable to $\log N$, (\ref{estS2Q}),  and Cauchy-Schwarz inequality, we have 
\begin{equation}
\lambda|G_\lambda|\leq C Q^{d/2} (\log Q)^{2d}|G_\lambda| + 
  \frac{C(N_1\cdots N_d)^{1/2} \max\{N_1, \cdots, N_d\}^{1+\e}    }{Q}|G_\lambda|^{1/2}\,.  
\end{equation}
For $\lambda\geq  C \max\{N_1, \cdots, N_d\}^{\frac{d}{2} +\e}$, take $Q$ to be a number satisfying  $ Q^{d/2}\max\{N_1, \cdots, N_d\}^\e = \lambda $ and then Lemma \ref{lemThm3} yields
\begin{equation}\label{estofG}
 |G_\lambda| \leq \frac{CN_1\cdots N_d \max\{N_1, \cdots, N_d\}^{2+\e}}{\lambda^{\frac{2(d+2)}{d}}}\,.
\end{equation}
Notice that 
\begin{equation}\label{L2ofS}
\|S_{N_1, \cdots, N_d}\|_2 \sim   \left( N_1\cdots N_d \right)^{1/2}\,. 
\end{equation}
Thus for $\lambda < C\max\{N_1, \cdots, N_d\}^{\frac{d}{2}+\e}$, we have 
\begin{equation}\label{estofG2}
|G_\lambda| \leq \frac{CN_1\cdots N_d}{\lambda^2} \leq \frac{C N_1\cdots N_d 
 \max\{N_1, \cdots, N_d\}^{2+\e}}{\lambda^{\frac{2(d+2)}{d}}}\,.
\end{equation}
Henceforth (\ref{estofG}) holds for all $\lambda >0$. We now 
estimate $L^{\frac{2(d+2)}{d}}$ norm of $S_{N_1,\cdots, N_d}$ by
\begin{equation}\label{LpnormofS}
\|S_{N_1, \cdots, N_d}\|^{\frac{2(d+2)}{d}}_{\frac{2(d+2)}{d}} \leq C\int_1^{2^{d} N_1\cdots N_d}
 \lambda^{\frac{2(d+2)}{d} -1 }|G_\lambda| d\lambda 
+ C\int_0^{1}
 \lambda^{\frac{2(d+2)}{d} -1 }|G_\lambda| d\lambda  \,. 
\end{equation}
Since (\ref{estofG}) holds for all $\lambda>0$, the first term in the right hand side 
of (\ref{LpnormofS}) can be bounded by $ CN_1\cdots N_d \max\{N_1, \cdots, N_d\}^{2+\e}$. The second term is clearly bounded by
$C$ because  $ G_\lambda $ is a set with finite measure. Putting both estimates together, 
we get
\begin{equation}\label{estofSnorm}
\|S_{N_1, \cdots, N_d}\|^{\frac{2(d+2)}{d}}_{\frac{2(d+2)}{d}}  \leq CN_1\cdots N_d \max\{N_1, \cdots, N_d\}^{2+\e}\,,
\end{equation}
as desired. Therefore, we complete the proof.\\

\section{Estimates of multi-linear maximal functions }\label{appl1}

In this section, we should provide an application of Theorem \ref{thm3}. 

\begin{definition}\label{defKadm}
Let $d\in \mathbb N$ and $K\in \{1, \cdots, d\}$. A subset $S$ of $\mathbb N^d$ is 
called $K$-admissible if for every element $(n_1, \cdots, n_d)\in S$, there 
exist $n_{i_1}, \cdots, n_{i_K}$ such that
\begin{itemize}
\item $i_1<i_2<\cdots< i_K$ and $i_1, \cdots, i_K\in \{1, \cdots, d\}  $;
\item $  \max\{n_1,\cdots, n_d \} \leq C\min\{n_{i_1},\cdots, n_{i_K}\} $. 
\end{itemize} 
Here the constant $C$ is independent of $(n_1,\cdots, n_d)$.
\end{definition}

\begin{theorem}\label{thm4}
Let $d, M_1, \cdots,  M_d\in\mathbb N$, $K\in \{1, \cdots, d\}$, and $A_{M_1, \cdots, M_d}$ be a multi-linear operator defined by
setting $A_{M_1, \cdots, M_d}(f_1, \cdots, f_{d+1})(n)$ to be 
\begin{equation}\label{defofAve}
 \frac{1}{M_1\cdots M_d} \sum_{m_1=1}^{M_1}\cdots\sum_{m_d=1}^{M_d}
 f_1(n-m_1)\cdots f_{d}(n-m_d)  
f_{d+1}\left(n-(m_1^2+\cdots+m_d^2) \right) \,. 
\end{equation}
Here $n\in \mathbb Z$.  Suppose $T^*$ is a maximal function given by
\begin{equation}\label{defofT*}
 T^*(f_1, \cdots, f_{d+1})(n)  = \sup_{(M_1, \cdots, M_{d})\in  S_K }
 \left| A_{M_1, \cdots, M_d}(f_1, \cdots, f_{d+1})(n)   \right|\,. 
\end{equation}
Here $S_K$ is any $K$-admissible subset of $\mathbb N^d$.  
Then if $K$ satisfies 
\begin{equation}\label{Kcond}
   K>\frac{2d}{d+4}\,,
\end{equation} 
then we have 
\begin{equation}\label{estofT*}
\left\|T^*(f_1, \cdots, f_{d+1})\right\|_{L^2(\mathbb Z)}\leq 
  C\prod_{j=1}^{d+1}\|f_j\|_{L^2(\mathbb Z)}\,. 
\end{equation}
Here $ L^2(\mathbb Z)$ stands for $L^2$ norm associated with counting measure on $\mathbb Z$,
 and $ C$ is independent of $f_j$'s but may depend on $K$ and $d$. 
\end{theorem}

\begin{remark}
Notice that for $d=1, 2, 3$, $\frac{2d}{d+4} <1$.  
Thus the condition (\ref{Kcond}) is superfluous in Theorem \ref{thm4} 
for $d=1, 2, 3$. Thus for $d=1,2, 3$, the set $S_K$ in Theorem \ref{thm4} 
can be replaced by $\mathbb N^d$ because $\mathbb N^d$ is $1$-admissible 
according to Definition \ref{defKadm}.
It is very possible  that, for $d\geq 4$,  the condition (\ref{Kcond}) on $K$ is redundant too. 
A delicate analysis involving the circle method should be utilized in order to remove (\ref{Kcond})
for the $d\geq 4$ cases.
We would not discuss this in this paper.   
\end{remark}

\begin{remark}
It is natural to ask whether the following inequality holds.
\begin{equation}
\left\|T^*(f_1, \cdots, f_{d+1})\right\|_{L^{\frac{2}{d+1}}(\mathbb Z)}\leq 
  C\prod_{j=1}^{d+1}\|f_j\|_{L^2(\mathbb Z)} ?
\end{equation}
This seems to be difficult but also be interesting.  So far we are only able to establish the boundedness of $T^*$ from $ L^2\times\cdots \times L^2$ to $L^p$ for $p> 2/(d+1)$ by an interpolation 
argument and Theorem \ref{thm4}. 
\end{remark}

To prove Theorem \ref{thm4}, we first introduce a simple  multi-linear estimate.

\begin{lemma}\label{lemApp}
Let $M\in\mathbb N$ and $F_1, \cdots, F_{M+1}$ be periodic functions on $\mathbb T$.
Let $T(F_1, \cdots, F_{M+1})$ be a multilinear operator given by
\begin{equation}\label{defofTF}
T(F_1, \cdots F_{M+1})(x_1, \cdots, x_M) = F_1(x_1)\cdots F_M(x_M) F_{M+1}(x_1+\cdots+x_{M})\,,
\end{equation} 
for $(x_1, \cdots, x_M)\in \mathbb T^M  $.
If $1\leq p \leq \frac{2M}{M+1}$,
\begin{equation}\label{estMulti}
\| T(F_1, \cdots F_{M+1})\|_{L^p(\mathbb T^M)} \leq \prod_{j=1}^{M+1}\|F_j\|_{L^2(\mathbb T)}\,.
\end{equation}
\end{lemma}
\begin{proof}
We only need to prove the case when $p=\frac{2M}{M+1}$, since other cases follow easily by 
H\"older equality. By a change of variables, we get
\begin{equation}\label{estLp0}
\| T(F_1, \cdots F_{M+1})\|_{L^p(\mathbb T^M)}\leq \|F_i\|_{\infty}\prod_{
  {\substack {j\neq i \\ j\in\{1, \cdots, M+1\} } }}\|F_j\|_p\,,
\end{equation}  
for any $ i \in \{1, \cdots, M+1\} $.
Now set $\alpha_1, \cdots, \alpha_{M+1}\in \mathbb Q^{M+1} $
by  $\alpha_1=(0, \frac{1}{p}, \cdots, \frac{1}{p})$,  $\alpha_2= (\frac{1}{p}, 0, \frac{1}{p},\cdots, \frac{1}{p})$, $\cdots$, $\alpha_{M+1}=(\frac{1}{p}, \cdots, \frac{1}{p}, 0)$. 
Clearly for $p=\frac{2M}{M+1}$, we have 
\begin{equation}
\left(\frac12, \cdots, \frac12\right) = \frac{1}{M+1}\left( \alpha_1 +\cdots +\alpha_{M+1}\right)\,.
\end{equation}
Thus $\left(\frac12, \cdots, \frac12\right) $ is in the convex hull generated by $ \alpha_1, \cdots, \alpha_{M+1}$. (\ref{estMulti}) follows immediately  by interpolation. 
\end{proof}

To finish the proof of Theorem \ref{thm4}, we need the following proposition.

\begin{proposition}\label{prop3}
Let $d\in \mathbb N$, $K\in \{1, \cdots, d\}$, $M_{K+1}, \cdots, M_d\in \mathbb N$.  Let
$A_{M, M_{K+1}, \cdots, M_d}$ be defined by setting $ A_{M, M_{K+1}, \cdots, M_d}(f_1, \cdots, f_{d+1})(n) $ to be
\begin{equation}\label{defofAMK}
\frac{1}{M^KM_{K+1}\cdots M_d }\left(\prod_{j=1}^K\sum_{m_j=1}^M \!\!\!f_j(n-m_j)\!\right)\left( \prod_{j=K+1}^{d}
\sum_{m_j=1}^{M_j} \!\!\!f_j(n-m_j) \!\right) f_{d+1}\left( n-(m_1^2+\cdots+m_d^2) \right)
\end{equation}
Suppose that $M\geq C\max\{M_{K+1},\cdots, M_d\}$. Then we have 
\begin{equation}
\|A_{M, M_{K+1}, \cdots, M_d} (f_1,\cdots, f_{d+1} )\|_{L^2(\mathbb Z)}\leq C\left(M_{K+1}\cdots M_d\right)^{-
\frac{d+4}{2(d+2)}}M^{\frac{-(d+4)K+2d}{2(d+2)}+\e} \prod_{j=1}^{d+1}\|f_j\|_{L^2(\mathbb Z)}\,.
\end{equation}
\end{proposition}

\begin{proof}
By duality, it is sufficient to prove that for any $f_{d+2}\in L^2(\mathbb Z)$, 
\begin{equation}\label{dual}
\sum_{n}A_{M, M_{K+1}, \cdots, M_d}(n) f_{d+2}(n) \leq C\left(M_{K+1}\cdots M_d\right)^{-
\frac{d+4}{2(d+2)}}M^{\frac{-(d+4)K+2d}{2(d+2)}+\e} \prod_{j=1}^{d+2}\|f_j\|_{L^2(\mathbb Z)} \,. 
\end{equation}
Now define $F_j$ for any $j\in\{1, \cdots, d+2\}$ by
\begin{equation}
 F_j(x) =\sum_{n} f_j(n) e^{2\pi i n x}\,. 
\end{equation} 
Then the left hand side of (\ref{dual}) can be represented by
\begin{equation}\label{represent}
\frac{1}{M^KM_{K+1}\cdots M_d }\int_{\mathbb T^{d+1}} \prod_{j=1}^{d+1} F_j(x_j) F_{d+2}(x_1+\cdots + x_{d+1}) S(x_1, \cdots, x_{d+1}) dx_1\cdots dx_{d+1}\,.
\end{equation}
Here $S(x_1, \cdots, x_{d+1})$ is given by
\begin{equation}
S(x_1, \cdots, x_{d+1})=  \sum_{m_1=1}^M\cdots\sum_{m_K=1}^M\sum_{m_{K+1}=1}^{M_{K+1}}
 \cdots\sum_{m_d=1}^{M_d} e^{2\pi i (m_1x_1+\cdots+ m_d x_d)} e^{2\pi i (m_1^2+\cdots + m_d^2)x_{d+1}}\,.
\end{equation}
Utilizing Theorem \ref{thm3}, we have 
$$
\|S\|_{\frac{2(d+2)}{d}}\leq C\left(M_{K+1}\cdots M_d\right)^{\frac{d}{2(d+2)}}M^{\frac{dK}{2(d+2)}+\frac{d}{d+2}+\e}\,.
$$ 
Then H\"older inequality yields that 
$$
{\rm (\ref{represent})} \leq C\|T(F_1, \cdots, F_{d+2})\|_{\frac{2(d+2)}{d+4}}\left(M_{K+1}\cdots M_d\right)^{-
\frac{d+4}{2(d+2)}}M^{\frac{dK}{2(d+2)} + \frac{d}{d+2}-K +\e}\,.
$$
Since $ \frac{2(d+2)}{d+4}\leq  \frac{2(d+1)}{d+2}$, we can apply Lemma \ref{lemApp} to obtain
\begin{equation}
{\rm (\ref{represent})} \leq  C\left(M_{K+1}\cdots M_d\right)^{-
\frac{d+4}{2(d+2)}}M^{\frac{-(d+4)K+2d}{2(d+2)}+\e} \prod_{j=1}^{d+2}\|F_j\|_{L^2(\mathbb T)}\,.
\end{equation}
\end{proof}

We now prove Theorem \ref{thm4}. Since $S_K$ is $K$-admissible, without loss of generality, we 
assume that $M_1=\cdots=M_K=M $ and $M\geq C\max\{M_{K+1}, \cdots, M_d\}$.
Moreover, we may also assume that $M$ is dyadic.  
 Henceforth we only need to
consider $\tilde T^*(f_1,\cdots, f_{d+1})$ given by
\begin{equation}
\tilde T^*(f_1,\cdots, f_{d+1}) =\sup_{M, M_{K+1}, \cdots, M_d}\left |A_{M, M_{K+1}, \cdots, M_d}(f_1, \cdots, f_{d+1})\right|\,.
\end{equation}
Clearly we have
$$
|\tilde T^*(f_1,\cdots, f_{d+1})|\leq \left(\sum_{M, M_{K+1}, \cdots, M_d}\left |A_{M, M_{K+1}, \cdots, M_d}(f_1, \cdots, f_{d+1})\right|^2\right)^{1/2}\,. 
$$
Taking $L^2$ norm for both sides, we then get
\begin{equation}
\|\tilde T^*(f_1,\cdots, f_{d+1})\|_{L^2(\mathbb Z)}\leq\left( \sum_{M, M_{K+1}, \cdots, M_d} 
\|A_{M, M_{K+1}, \cdots, M_d}(f_1, \cdots, f_{d+1})\|_{L^2(\mathbb Z)}^2  \right)^{1/2}\,.
\end{equation}
Employing Proposition \ref{prop3}, we estimate $\|\tilde T^*(f_1,\cdots, f_{d+1})\|_{L^2(\mathbb Z)}$ by
\begin{equation}
 \left( \sum_{M, M_{K+1}, \cdots, M_d}
  C\left(M_{K+1}\cdots M_d\right)^{-
\frac{d+4}{(d+2)}}M^{\frac{-(d+4)K+2d}{(d+2)}+\e}  \right)^{1/2} \prod_{j=1}^{d+1}\|f_j\|_{L^2(\mathbb Z)}\,,
\end{equation}
which is bounded by 
$$
C\prod_{j=1}^{d+1}\|f_j\|_{L^2(\mathbb Z)}\,,
$$
since $K>\frac{2d}{d+4}$ implies $ \frac{(d+4)K-2d}{(d+2)} >0$. This completes the proof of Theorem \ref{thm4}.

A similar argument yields Theorem \ref{thm5}. We omit its proof.

\begin{theorem}\label{thm5}
Let $d \in \mathbb N $, $N\in\mathbb N$, and $A_{N}$ be a multi-linear operator defined by
setting $A_{N}(f_1, \cdots, f_{d+1})(n)$ to be 
\begin{equation}\label{defofAve1}
 \frac{1}{N^d} \sum_{m_1=1}^{N}\cdots\sum_{m_d=1}^{N}
 f_1(n-m_1)\cdots f_{d}(n-m_d)  
f_{d+1}\left(n-(m_1^2+\cdots+m_d^2) \right) \,. 
\end{equation}
Here $n\in \mathbb Z$.  Suppose $T^*$ be a maximal function given by
\begin{equation}\label{defofT*1}
 T^*(f_1, \cdots, f_{d+1})(n)  = \sup_{N \in  \mathbb N }
 \left| A_{N}(f_1, \cdots, f_{d+1})(n)   \right|\,. 
\end{equation}
Then we have 
\begin{equation}\label{estofT*1}
\left\|T^*(f_1, \cdots, f_{d+1})\right\|_{L^2(\mathbb Z)}\leq 
  C\prod_{j=1}^{d+1}\|f_j\|_{L^2(\mathbb Z)}\,. 
\end{equation}
\end{theorem}

Also we are able to obtain $L^2$ estimate for the corresponding bilinear Hilbert transform.

\begin{theorem}\label{thm6}
Let $K$ be a function on $\mathbb Z$ satisfying 
\begin{equation}
 \left| K(n) \right|\leq \frac{C}{|n|}  
\end{equation}
for $n\neq 0$. Let $T(f_1, f_2)$ be defined by
\begin{equation}\label{defofTf1}
T(f_1, f_2)(n)=\sum_{m\neq 0} K(m)f_1(n-m)f_2(n-m^2)\,,  
\end{equation} 
for Schwartz functions $f_1, f_2: \mathbb R\mapsto \mathbb C$. 
Then we have
\begin{equation}\label{Testf1}
 \|T(f_1, f_2)\|_{L^2(\mathbb Z)} \leq C\|f_1\|_{L^2(\mathbb Z)}\|f_2\|_{L^2(\mathbb Z)}\,. 
\end{equation}
\end{theorem}

\begin{proof}
For any dyadic number $M\geq 1$, define $T_M(f_1, f_2)$ by 
\begin{equation}
T_M(f_1, f_2)(n) = \frac{1}{M}\sum_{m\sim M}|f_1(n-m)f_2(n-m^2)|\,. 
\end{equation}
Apply Proposition \ref{prop3} to get
\begin{equation}\label{TM1}
\|T_M(f_1, f_2)\|_{L^2(\mathbb Z)}\leq M^{-1/2+\e}\|f_1\|_{L^2(\mathbb Z)}\|f_2\|_{L^2(\mathbb Z)}\,.
\end{equation}
(\ref{Testf1}) follows from (\ref{TM1}). 
\end{proof}

\begin{remark}
If the kernel $K$ in Theorem \ref{thm6} has some cancellation condition, then $T(f_1, f_2)$ could be 
a bounded operator from $ L^2\times L^2$ to $L^1$. This problem is still open and seems to be  challenging.
\end{remark}

\vspace{0.1cm}

\section{Estimate for $K_{p,d,N}$ when $p$ is even}

In this section, we give a proposition on $K_{p,d, N}$ when $p$ is even. The idea is not new, and 
it is utilized often in the field of number theory. For the sake of self-containedness, we include 
it here. By using it and an arithmetic argument, one can get sharp estimates, up to a factor of $N^\e$, for $K_{6,1, N}$, $K_{4, 2, N}$, etc. See \cite{B1} for details.  

\begin{proposition}\label{prop1}
If $p>0$ is an even integer, then we have 
\begin{equation}\label{estKpeven}
 K^p_{p,d,N}\leq \sup_{(\bl, m)\in S_{d, pN/2}\times \{1, \cdots, pN^2/2\}
  } e^{2\pi \e m} \mathcal F_{\mathbb T^{d}\times \mathbb T}(F^{p/2}(\cdot, \cdot +i\e))(\bl, m) \,. 
\end{equation}
Here $\mathcal F_{\mathbb T^{d}\times \mathbb T}$ is Fourier transform of functions on $\mathbb T^{d} \times \mathbb T$, $\e$
is any positive number, 
and $F$ is  given by
\begin{equation}\label{defofFz}
 F(\bx, z)= \sum_{\bn\in \mathbb Z^d} e^{2\pi i z|\bn|^2 + 2\pi i \bx\cdot \bn}\,.
\end{equation}

\end{proposition}

\begin{proof}
Let $k=p/2$. A direct calculation yields 
\begin{equation}
 \int_{\mathbb T^{d+1}} \left| \sum_{\bn\in S_{d,N}} a_{\bn}e^{2\pi i (\bn\cdot \bx+ |\bn|^2t)}   \right|^{2k}  d\bx dt
 = \sum_{(\bn_1, \cdots, \bn_k, \bm_1, \cdots, \bm_k)\in S_{d,N,k}}a_{\bn_1}\cdots a_{\bn_k}\overline{a_{\bm_1}}
\cdots \overline{a_{\bm_k}} \label{even15}
 \,.
\end{equation}
Here $S_{d,N,k}$ is given by
$$
S_{d,N,k}=\left\{(\bn_1, \cdots, \bn_k, \bm_1, \cdots, \bm_k)\in S_{d,N}^{p}: \sum_{j=1}^k\bn_j=\sum_{j=1}^k\bm_j\,\,, 
 \sum_{j=1}^k|\bn_j|^2=\sum_{j=1}^k |\bm_j |^2\right\}
$$
For any $\bl\in S_{d, kN}$ and any positive integer $m\leq kN^2 $, we set
$$
S_k(\bl, m)=\left\{(\bn_1, \cdots, \bn_k)\in S_{d,N}^{k}: \sum_{j=1}^k\bn_j=\bl, \,\,\sum_{j=1}^k|\bn_j|^2=m \right\}\,.
$$
  We now can estimate 
(\ref{even15}) by
\begin{equation}\label{16}
 \sum_{\bl\in S_{d, kN}}\sum_{m=1}^{kN^2} \left|\sum_{(\bn_1, \cdots, \bn_k)\in S_k(\bl,m)}a_{\bn_1}\cdots a_{\bn_k} \right|^2  
\end{equation}
Utilizing Cauchy-Schwarz inequality and the fact that $\{S_k(\bl, m)\}$ forms a partition of $S_{d,N}^k$, we 
dominate (\ref{16}) by 
\begin{equation}
 \max_{\bl\in S_{d, kN}, 1\leq m\leq kN^2}\left|S_{k}(\bl, m)\right| \left( \sum_{\bn}|a_\bn|^2\right)^k\,,
\end{equation}
where $\left|S_{k}(\bl, m)\right|$ denotes the cardinality of $S_{k}(\bl, m)$.


Employing the elementary fact $\int_0^1 e^{2\pi i n\theta} d\theta= 0$ if $n\neq 0$ and
$\int_0^1 e^{2\pi i n\theta} d\theta= 1$ if $n=0$, for any $\bl\in S_{d, kN}$ and any positive integer
$m\leq kN^2$, we can estimate $|S_k(\bl,m)|$ by
\begin{equation}\label{19}
\sum_{(\bn_1, \cdots, \bn_k)\in S_{d, N}^k  }
 \int_0^1 e^{2\pi i t(\sum_{j=1}^k |\bn_j|^2-m )} dt  \int_{\mathbb T^{d}} e^{2\pi i\sum_{j=1}^k
  \bx\cdot \bn_j} e^{-2\pi i \bx\cdot \bl }d\bx\,,
\end{equation}
which equals to
\begin{equation}\label{110}
\sum_{(\bn_1, \cdots, \bn_k)\in S_{d, N}^k  }
  e^{2\pi \e m }\int_0^1 e^{2\pi i (t+i\e)\sum_{j=1}^k |\bn_j|^2}  e^{-2\pi im t}dt  \int_{\mathbb T^{d}} e^{2\pi i\sum_{j=1}^k
  \bx\cdot \bn_j} e^{-2\pi i \bx\cdot \bl }d\bx\,,
\end{equation}
for any real number $\e$. This term can also be written as
\begin{equation}\label{111}
e^{2\pi \e m }\int_{\mathbb T^{d}\times \mathbb T}
\left(\sum_{\bn\in S_{d, N}} e^{2\pi i (t+i\e)|\bn|^2} e^{2\pi i \bx\cdot \bn }  \right)^k e^{-2\pi i \bx\cdot \bl}
e^{-2\pi i m t} d\bx dt\,.
\end{equation}
Notice that we may replace $S_{d, N}$ by $\mathbb Z^d$ in (\ref{19}), (\ref{110}) and (\ref{111}) 
to make upper bounds larger. Thus,
by the definition of $F$ in (\ref{defofFz}), we dominate $|S_k(\bl,m)|$ by
\begin{equation}
  e^{2\pi \e m} \int _{\mathbb T^{d}\times \mathbb T}\left( F( \bx, t+i\e)\right)^k e^{-2\pi i \bx\cdot \bl}
e^{-2\pi i m t} d\bx dt\,.
\end{equation}
This finishes the proof of Proposition \ref{prop1}.

\end{proof}

{\bf Acknowledgement}. The first author wishes to thank his
advisor Xiaochun Li, for his valuable and insightful suggestions and enthusiastic guidance.

\end{document}